\documentclass{amsart}

\usepackage{amsmath, amssymb,epsfig, verbatim}
\usepackage{graphicx}
\usepackage{epic}
\usepackage{color}

\newtheorem{thm}{Theorem}[section]

\newtheorem{cor}[thm]{Corollary}
\newtheorem{lem}[thm]{Lemma}
\newtheorem{prop}[thm]{Proposition}

\newtheorem{defn}[thm]{Definition}

\newtheorem{rem}[thm]{Remark}

\numberwithin{equation}{section}

\newcommand{\Z}{\mathbb Z}
\newcommand{\N}{\mathbb N}
\newcommand{\Q}{\mathbb Q}

\newcommand{\CP}{{\mathbb C}{\mathbb P}}
\newcommand{\cpkk}{{\overline {{\mathbb C}{\mathbb P}^2}}}
\newcommand{\cpk}{{\mathbb {CP}}^2}

\begin{document}

\title{Definite four-manifolds with exotic smooth structures}

\author{Andr\'{a}s I. Stipsicz}
\address{HUN-REN R\'enyi Institute of Mathematics\\
H-1053 Budapest\\ 
Re\'altanoda utca 13--15, Hungary}
\email{stipsicz.andras@renyi.hu}

\author{Zolt\'an Szab\'o}
\address{Department of Mathematics\\
Princeton University\\
 Princeton, NJ, 08544}
\email{szabo@math.princeton.edu}

\begin{abstract}
  In this paper we study smooth structures on closed oriented
  $4$-manifolds with fundamental group $\Z/ 2\Z$ and definite
  intersection form. We construct infinitely many irreducible, smooth,
  oriented, closed, definite four-manifolds with $\pi_1=\Z/2\Z$ and
  $b_2=1$, and $b_2=2$.  As an application, we prove that when the
  second Betti number $b_2$ of a definite four-manifold with $\pi
  _1=\Z/2\Z$ is positive and it admits a smooth structure, then it
  admits infinitely many smooth structures.
\end{abstract}
\maketitle

\section{Introduction}
\label{sec:intro}
In \cite{LLP} an exotic four-manifold with definite intersection form
and non-trivial fundamental group was constructed.  In~\cite{definite}
further such examples have been found, and on smooth,
oriented, closed, definite four-manifolds with fundamental group
$\Z/2\Z$ and Euler characteristic $\geq 6$ infinitely many exotic
smooth structures were given. The construction of \cite{definite}
relied on finding exotic smooth structures on $\cpk \#9\cpkk$ with
fixed point free involutions, and used Seiberg-Witten invariants and
their transformation rule under the blow-up procedure. In this paper
we extend these studies and find exotic smooth structures on other
simply connected four-manifolds with fixed point free involutions,
leading to our main result:
\begin{thm}\label{thm:main}
  Suppose that $X$ is a closed, oriented, topological four-manifold
  with $\pi _1(X)=\Z /2\Z$, with definite intersection form and with
  second Betti number $b_2(X)>0$.  If $X$ admits a smooth structure,
  then it admits infinitely many non-diffeomorphic smooth structures.
\end{thm}
In comparison, no closed simply connected definite smooth four-manifold
with an exotic smooth structure is known.

While in \cite{definite} the constructions rested on the rational
blow-down method, combined with knot surgery and double node surgery,
in the proof of Theorem~\ref{thm:main} we will mostly rely on
Luttinger surgery and, more generally, torus surgery~\cite{FPS}.
In particular, we will prove the following two theorems. To state
them, let us define the closed, oriented, smooth four-manifold $Z_1$
by quotienting $S^2\times S^2$ with the orientation preserving
involution given by the antipodal maps on the two factors. It is easy
to see that $\pi _1(Z_1)=\Z/2\Z$ and $b_2(Z_1)=0$.

\begin{thm}\label{thm:b2=2}
  The smooth four-manifold $Z_1\#2\cpkk$ admits infinitely many
  irreducible exotic smooth structure.
\end{thm}
Recall that a smooth, closed four-manifold $X$ is irreducible, if a
smooth decomposition $X=X_1\# X_2$ into connected sum implies that
either $X_1$ or $X_2$ is homeomorphic to the four-sphere $S^4$.

\begin{thm}\label{thm:b2=1}
  The smooth four-manifold $Z_1\#\cpkk$ admits infinitely many
  irreducible exotic smooth structure.
\end{thm}

\begin{rem}
  At the time of writing the manifolds given by Theorem~\ref{thm:b2=1} have the smallest Euler characteristic among known exotic
  smooth closed oriented four-manifolds with finite fundamental groups,
  or indeed with
  $b_1=0$.
  \end{rem}

The technique applied in the proof of Theorem~\ref{thm:b2=1} allows us to
show that by blowing up the exotic manifolds they stay exotic leading
to the proof of Theorem~\ref{thm:main}. Combining with earlier
results in \cite{definite} for the $b_2=4$ case we immediately have

\begin{cor}\label{cor:irr}
  Let $X$ be a smooth, closed, oriented, topological four-manifold  with
  $\pi_1(X)=\Z /2\Z$, with definite intersection form and with second Betti
number $b_2(X)=1,2,4$. Then $X$ admits infinitely many non-diffeomorphic
irreducible smooth structures. \qed
\end{cor}

For other values of $b_2$, the existence of irreducible examples in
Corollary~\ref{cor:irr} are still open. In this paper we use double
covers and show that $X$ is exotic and irreducible since its double
cover is exotic and irreducible.  When $b_2>4$, one would need
irreducible manifolds homeomorphic to $\cpk\# k\cpkk$ for $k =
1+2b_2>9$ with a free $\Z/2\Z$ symmetry.  The existence of an
irreducible smooth structure on $\cpk \# k \cpkk$ with $k>9$ is,
however, open (even without adding the symmetry property).  Similarly
in the $b_2=0$ case (see Remark~\ref{rem:b2=0}), one would need an
exotic $S^2\times S^2$ with a free $\Z/2\Z$ symmetry, and even the
existence of an exotic $S^2\times S^2$ is unknown. The $b_2=3$ case is
more promising: at least in this case there are plenty of known
irreducible smooth structures on $\cpk\# 7\cpkk$; one only needs to
find the appropriate free involutions on them. We hope to return to
this question in future work.

The paper is organized as follows. In Section~\ref{sec:background} we
discuss some preliminaries. In Section~\ref{sec:b2=2} we study
examples with $b_2=2$, leading to the proof of Theorem~\ref{thm:b2=2},
and in Section~\ref{sec:b2=1} we study examples with $b_2=1$,
and prove Theorem~\ref{thm:b2=1} and ultimately Theorem~\ref{thm:main}.

{\bf Acknowledgements}: The first author was partially supported by
the \'Elvonal (Frontier) grant KKP144148 of the NKFIH. The second
author was partially supported by NSF grant DMS-1904628 and the Simons
Grant {\it New structures in low dimensional topology. }

\section{Background}
\label{sec:background}

This section is devoted to collect some of the background results we
will need in our subsequent discussions.

\subsection{Topological classification of definite four-manifolds with
  fundamental group $\Z/2\Z$}

Suppose that $X$ is a closed, oriented four-manifold with finite
cyclic fundamental group, which admits a smooth structure, and let
$\widetilde{X}$ denote the universal cover of $X$.  According to
\cite[Theorem~C]{HK}, such manifolds are classified up to
homeomorphism by their
intersection form on $H_2(X; \Z)/Tors$, and whether $X$ and
$\widetilde{X}$ are spin or not. (Notice that if $X$ is spin, then
$\widetilde{X}$ is also spin.)
Combining this result with theorems of Donaldson, we get

\begin{thm}\label{thm:TopClassification}
  The smooth manifolds $X_1$ and $X_2$ with fundamental group
  $\Z/2\Z$, negative definite intersection form and second
  Betti number $b_2>0$ are homeomorphic if and only if
  their Euler characteristics are equal.
\end{thm}

\begin{proof}
  By Donaldson's diagonalizability theorem~\cite{Don2},
   the intersection forms of
  $X_1$ and $X_2$ depend only on the size of $b_2$, hence if the Euler
  characteristics are equal, the intersection forms are isomorphic.
  When $b_2>0$, the diagonalizability theorem also implies that the
  four-manifolds are non-spin. The universal (double) cover
  $\widetilde{X_i}$ is also non-spin, as $b_2^+(\widetilde{X_i})$ is
  equal to 1, but (once again by $b_2(X_i)>0$) we have
  $b_2^-(\widetilde{X_i})>1$, contradicting Donaldson's Theorem~B (see
  \cite[Theorem~1.3.2]{DK}).  Therefore both $X_i$ and
  $\widetilde{X_i}$ are non-spin for $i=1,2$, hence
  \cite[Theorem~C]{HK} outlined above implies the result.
  \end{proof}

\begin{rem}\label{rem:b2=0}
  The case $b_2=0$ is special: in this case there are
  manifolds $Z_0, Z_1$ (as explained in \cite{definite}, see also
  \cite{Hil}) with trivial intersection form, $\pi _1=\Z/2\Z$ and
  $Z_0$ spin, $Z_1$ non-spin. (Their universal cover is $S^2\times
  S^2$.) According to \cite[Lemma~12.3(2)]{Hil} the universal cover of
  an oriented, closed, smooth four-manifold with $\pi _1=\Z/2\Z$ and
  Euler characteristic $2$ is homeomorphic to $S^2\times S^2$, so
  these are the only two examples up to homeomorphism.  The existence
  of exotic structures on $Z_0,Z_1$ (and indeed on $S^2\times S^2$) is
  still open.
\end{rem}

\subsection{Torus surgeries}
We briefly recall the concept of surgery along a torus in a
four-manifold from \cite{FPS}.
Suppose that $X$ is a smooth four-manifold and $T\subset X$ is an
oriented torus with 0 self-intersection number (that is, with trivial
normal bundle).  Let $f$ be a trivialization of the normal bundle $\nu
T$, that is, a \emph{framing} on $T$.  Using such a trivialization
$f$, an oriented simple closed curve $C\subset T$ gives rise to an
oriented simple closed curve $c_f\subset \partial (\nu
T)=T^3$. Furthermore, let $\mu _T$ denote the (oriented) meridian of
$T$, considered also in $\partial (\nu T)$.  For a rational number
$\frac{p}{q}$, let $(X, T, f, \gamma , \frac{p}{q})$ denote the
four-manifold we get by deleting ${\rm {int}}\, \nu T$, and gluing
back $T^2\times D^2$ (the union of a four-dimensional 2-handle, two
3-handles and a 4-handle) by attaching the 2-handle along a simple
closed curve representing the homology class $p[\mu _T] + q [c_f]$.
This construction is known as the \emph{torus surgery} on $T$ along
the curve $C$ with framing $f$ and coefficient $\frac{p}{q}$.

\subsection{Luttinger surgeries}
When the four-manifold $X$ admits a symplectic structure
$\omega$ and the torus $T\subset (X, \omega )$ is Lagrangian (i.e., $\omega
\vert _T=0$), the torus  $T$ comes with a canonical (Lagrangian) framing:
By the Lagrangian neighbourhood
theorem an appropriate small normal neighbourhood of $T$ is
symplectomorphic to a neighbourhood of the zero-section of the
cotangent bundle $T^*T^2\to T^2$ of the 2-torus $T^2$.
%(when the total
%space of this bundle is
%equipped with the tautological symplectic form coming from its
%Liouville 1-form).
As $T^*T^2$ is trivialized, this identification equips $T\subset X$
with its \emph{Lagrangian framing}.  Indeed, the push-off of $T$ along
this framing is also Lagrangian.  According to \cite{ADK}, torus
surgery on a Lagrangian torus along any simple closed curve, with the
Lagrangian framing and with surgery coefficient $\frac{1}{k}$ ($k\in
\Z$) results a four-manifold $(X, T, \gamma , \frac{1}{k})$ with a
symplectic structure, which is equal to $\omega$ outside of a suitably
chosen neighbourhood of $T$.  This special case of torus surgery is
called \emph{Luttinger surgery} \cite{ADK}.

\subsection{Seiberg-Witten invariants}
In this paper we will distinguish smooth structures on 4-manifolds
(with fundamental group $\Z /2\Z$) by showing that their universal
covers are non-diffeomorphic. For distinguishing those simply connected
manifolds we will apply Seiberg-Witten invariants.

Recall that the Seiberg-Witten function on a smooth, closed,
oriented four-manifold $X$ is a map
\[
SW_X\colon Spin^c(X) \to \Z.
\]
For manifolds with $b_2^+(X)>1$ or with $b_2^+(X)=1$ and $b_2^-(X)\leq
9$ this map is --- up to sign --- a diffeomorphism invariant, that is,
for a diffeomorphism $f\colon X_1\to X_2$ we have $SW_{X_2}(K)=\pm
SW_{X_1}(f^*(K))$. 

When $H_1(X; \Z )$ has no $\Z/ 2\Z$ torsion, $spin^c$ structures ${\bf s}$
on $X$ are naturally identified with characteristic cohomology classes
in $H^2(X; \Z)$, that is, with  classes whose mod 2 reductions agree with
$w_2(X)$. The map is given by associating to ${\bf s}$ its first Chern
class $c_1({\bf s})$.

The Seiberg-Witten function $SW_X$ has finite support;
elements of the support of $SW_X$ are called \emph{basic classes} of
$X$. Furthermore $SW_X(-K)=\pm SW_X(K)$, and it satisfies the
adjunction inequality: for an embedded, closed, oriented surface
$\Sigma \subset X$ of genus $g(\Sigma)>0$ with self-intersection
$[\Sigma ]^2$ non-negative, we have
\[
2g(\Sigma )-2\geq [\Sigma ]^2+\vert K([\Sigma ])\vert
\]
once $SW_X(K)\neq 0$.
A further important property of $SW_X$ is that for a symplectic
four-manifold $(X, \omega )$
(as above) with $b_2^+(X)>1$
we have
\begin{equation}\label{eq:symp}
SW_X(\pm c_1(X, \omega ))=\pm 1.
\end{equation}

\begin{lem}
  If a simply connected smooth four-manifold $X$ (with $b_2^+(X)>1$ or
  $b_2^+(X)=1$ and $b_2^- \leq 9$) has at least one Seiberg-Witten
  basic class
  $\{K_1, ..., K_n\}$  and 
  $(K_i - K_j)^2$ is never $-4$, then $X$ is irreducible.
\end{lem}
\begin{proof}
  Suppose that $X$ has a connected sum decomposition as $X= X_1\#X_2$.
  The non-vanishing of the Seiberg-Witten invariant $SW_X$ of $X$
  implies that one of $X_1$ or $X_2$ has negative definite
  intersection form; suppose it is $X_2$.  Donaldson's
  diagonalizability theorem \cite{Dona, Do, Don2} implies that the
  intersection form of $X_2$ is diagonalizable. If $E_1,\ldots ,E_k$
  is a diagonal basis (with $E_i^2=-1$), then Seiberg-Witten basic
  classes of $X$ are given by $K'_j \pm E_1 \ldots \pm E_k$ where $K'_j$
  are the basic classes of $X_1$. As $(K_i-K_j)^2$ cannot be equal to
  $-4$ for any choice of $i,j$, it follows that $k=0$, and so $X_2$ is
  a homotopy $S^4$ and so it is homeomorphic to $S^4$ by Freedman's
  theorem.
  \end{proof}

\subsection{Seiberg-Witten invariants and torus surgeries}
The transformation of the Seiberg-Witten invariant under torus surgery
can be determined using the formula from \cite{MMSz}. For the version
we need, we fix a smooth, oriented, compact four-manifold $M$ with
three-torus boundary $\partial M = T^3$, and a spin$^c$ structure {\bf
  s} whose first Chern class restricts trivially to $\partial M$. Fix
two homology classes $a, b \in H_1(T^3)$ given by $a= [S^1\times
  pt\times pt]$ and $b = [pt\times S^1\times pt]$. Given a pair of
integers $p,q$, where $(p,q) \in \Z\times \Z$ is a primitive element,
let $M(p,q)$ denote the effect of torus surgery, where we glue $M$ and
$D^2\times T^2$ along $\phi$ to get $M(p,q) = M\cup _\phi D^2\times
T^2$, so that in homology $\phi _{\ast}$ maps $[\partial (D^2)]$ to
$pa+qb$. Let $S(p,q)$ denote the set of spin$^c$ structures on
$M(p,q)$ whose restriction to $M$ agrees with ${\bf s}$ and define
$F(p,q)$ as the sum of the Seiberg-Witten invariants of classes in
$S(p,q)$. Then, according to \cite{MMSz} we have
\begin{equation}\label{eq:formula}
  F(p,q)= p\cdot F(1,0)+ q\cdot F(0,1).
\end{equation}

\section{Exotic structures on definite manifolds with $b_2=2$}
\label{sec:b2=2}

By finding a genus-2 surface with self-intersection 0 in the twice
blown-up four-torus $T^4\# 2\cpkk$, taking normal connected sum of two
copies and applying four appropriate torus surgeries (the last two
depending on an integer parameter $n$) we get a sequence of manifolds
$X_n$ --- the details of this construction is given in
Subsection~\ref{subsec:cons}. Determining the fundamental group of
$X_n$ and its characteristic numbers, we deduce that these
four-manifolds are homeomorphic to $\cpk \# 5\cpkk$, while the
construction equips them with a natural involution. The determination
of the Seiberg-Witten functions $SW_{X_n}$ then shows that these
manifolds are smoothly distinct, leading to the proof of
Theorem~\ref{thm:b2=2}.  The topological argument is given in
Subsection~\ref{ssec:topological}, while the Seiberg-Witten
calculation is detailed in Subsection~\ref{ssec:computeSW}.

\subsection{The construction of the four-manifolds}
\label{subsec:cons}
We start with a construction of definite manifolds with $\pi
_1=\Z/2\Z$ and second Betti number equal to 2 (so Euler characteristic
4). The idea is similar to the one applied in \cite{definite}: we
construct exotic smooth structures on $\cpk \# 5\cpkk$ which admit
orientation preserving free involutions. One building block in this
construction will be a family of smooth oriented closed 4-manifolds
that are constructed from $T^4$ by doing torus surgeries along two
disjoint Lagrangian tori $T_1$ and $T_2$. Such constructions are
studied in \cite{AP, BK}; in the following
exposition we will use several results from \cite{BK}.

Let $T^4$ be given as the quotient of the four-dimensional unit cube
$[0,1]^4$. We use coordinates $x= (z_1,z_2,z_3,z_4)$ to describe
points $x \in T^4$.  Fix four real numbers $c_1,c_2,c_4', c_4 $ in
$[1/2,3/4]$ with $c_4'<c_4 $.  The torus $T_1$ is given by the
equations $z_2=c_2$, $z_4 = c_4$, while the torus $T_2$ is given by
$z_1= c_1$, $z_4= c_4'$.  Given the product symplectic form on $T^4 =
T^2\times T^2$, both $T_1$ and $T_2$ are Lagrangian.

The maps from $[0,1]$ to $T^4$  given by
$x=(t,0,0,0)$, $y = (0,t,0,0)$, $a=(0,0,t,0)$, $b = (0,0,0,t)$, with
$t\in [0,1]$ are loops based at $x_0= (0,0,0,0)$. 

We also fix some further notations:
Given a diagonal point
$(t,t) \in T^2$ and a radius $0<r<1/2$ we denote by
$D_{t}(r)$ the closed disk of radius $r$ around $(t,t)$. 
Choose an $\epsilon> 0$ so that for ${\bf D}= D_{2\epsilon}(4\epsilon)$ the
submanifold
\begin{equation}\label{eq:m0}
  M_0 = ({\bf D}\times T^2)\cup (T^2\times {\bf D})
\end{equation}
is disjoint from both $T_1$ and $T_2$.  Let $H=K=T^2-D_{2\epsilon}(1.9
\epsilon)$. Note that $H\times K\subset T^4$ contains the fixed tori
$T_1$ and $T_2$, and also the four based loops, $x,y,a,b$.

  Choose small closed tubular neighborhoods of $T_1$ and $T_2$ in
  $T^4$ that are disjoint from $M_0$ of Equation~\eqref{eq:m0}. The
  boundaries of the neighbourhoods are three-dimensional tori $T^3_1$
  and $T^3_2$. It has been shown in \cite[Theorem~2]{BK} that there are arcs
  inside $H\times K$ from $x_0$ to $T^3_1$ and $T^3_2$ so that the
  meridian $\mu _1$ of $T_1$ and the Lagrangian push-offs of the generating
  loops $m_1,\ell_1$ of $H_1(T_1; \Z )$ are given by
  \[
  \mu_1=[b^{-1},y^{-1}], \qquad m_1=x,  \qquad \ell_1= a,
  \]
while the meridian $\mu _2$ of $T_2$ and the Lagrangian push-offs of
the generating loops $m_2,\ell_2$ of $H_1( T_2; \Z )$ are given by
\[
\mu_2=[x^{-1},b], \qquad m_2= y,\qquad \ell_2= bab^{-1}.
\]
Let us denote the homology classes of the meridians and the  push-offs in
$H_1(T^3_i; \Z )$ by $[\mu_i]$, $[m_i]$ and $[\ell_i]$.

Given integers $p_1,q_1$, $p_2,q_2$, where $(p_1,q_1)=1$ and
$(p_2,q_2)=1$, we perform torus surgeries along $T_1$ and $T_2$ with
surgery slopes $p_1[\mu_1]+q_1[m_1]$ and $p_2 [\mu_2]+q_2[\ell_2]$
respectively.  Let $V_0(p_1/q_1,p_2/q_2)$ denote the effect of this
surgery in $H\times K$ and let $V(p_1/q_1,p_2/q_2)$ denote the effect
of this surgery on $T^4$; so $V_0(p_1/q_1,p_2/q_2)\subset
V(p_1/q_1, p_2/q_2)$.  Note that $(p,q)=(1,-1)$ correspond to $-1$
Luttinger surgery, so $V(-1,-1)$ is a symplectic manifold.  Note also
that $M_0\subset V(p_1/q_1,p_2/q_2)$.

Now we construct a symplectic genus-2 surface in $M_0 \subset
V(p_1/q_1,p_2/q_2)$ taking the transversally intersecting symplectic
tori $T^2\times (\epsilon, \epsilon)$ and $(\epsilon,\epsilon)\times
T^2$ and smoothing the intersection point in the $\epsilon/3$
neighborhood of $(\epsilon,\epsilon, \epsilon, \epsilon)$. We further
blow up this surface at two points
$k_1=(1/2,1/2, \epsilon, \epsilon)$ and $k_2=(\epsilon, \epsilon,1/2,1/2)$.
This gives a symplectic
genus-2 surface $G_2$ with self-intersection $0$ in the
closed four-manifold
\begin{equation}\label{eq:defofw2}
W_2(p_1/q_1,p_2/q_2)= V(p_1/q_1,p_2/q_2)\# 2\cpkk .
\end{equation}
(The index 2 is supposed to indicate that we applied two blow-ups.)
%The following result will be used later for Seiberg-Witten
%calculations.

Now choose a small enough closed tubular neighborhood $N_2$ of $G_2$
and fix a trivialization on $N_2$ as $G_2\times D^2$. In the following
we will fix some loops on $\partial N_2$.  First, let $M_{x,y}$ denote
$T^2\times {\bf D}$ blown up at $k_1$ and let $M_{a,b}$ denote ${\bf
  D}\times T^2$ blown up at $k_2$.  Then choose a basepoint $x_1$ on
$G=G_2\times \{ p\}\subset G_2\times D^2$ where $ p\in \partial
D^2$. Finally choose five embedded loops $s_1, t_1,s_2, t_2,\mu$ on
$G_2\times \{ p\}$ based at $x_1$ (see Figure~\ref{fig:maps}) with the
following properties:

\begin{itemize}
\item $[s_1,t_1] = \mu$
\item $[s_2,t_2]= \mu^{-1}$
\item $\mu$ is inside ${\bf D}\times {\bf D}$
\item $s_1,t_1$ are inside $M_{x,y}$ and $[s_1]=[x]$, $[t_1]=y$ in
  $H_1(M_{x,y}; \Z)$.
\item $s_2,t_2$ are inside $M_{a,b}$ and $[s_2]=[a]$, $[t_2]=b$ in
  $H_1(M_{a,b}; \Z)$.
\end{itemize}

Note that $G_2\times \{ p\}$ is embedded in $\partial (W_2(p_1/q_1,
p_2/q_2)-int(N_2))$.  Now we define a fixed point free, orientation
reversing diffeomorphism $\psi$ on $G$ with $\psi\circ\psi ={\rm
  {id}}_G$.  We take $\psi$ to be the composition two maps:  The first
map is the orientation reversing involution $r$ which swaps the two
torus components, with fixed points given by $\mu$ and mapping the curve
$s_1$ to $t_2$ and $t_1$ to $s_2$.  The second is the hyperelliptic
involution $h$ (180$^\circ$ rotation), as shown by
Figure~\ref{fig:maps}. (Here we fix a particular identification of $G$
with the model genus-2 surface, described by Figure~\ref{fig:maps}.)

\begin{figure}[htb]
\begin{center}
\setlength{\unitlength}{1mm}
 \includegraphics[height=9cm]{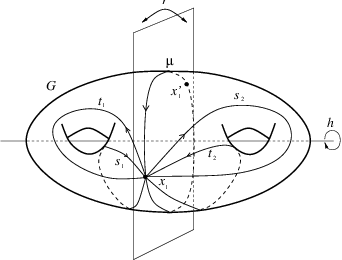}
\end{center}
\caption{\quad The map $r$ is reflection to the plane depicted (hence
  reverses the orientation of the genus-2 surface $G$), while
  $h$ is 180$^\circ$ rotation around the axis shown.}
\label{fig:maps}
\end{figure}

We will define $P(p_1/q_1, p_2/q_2,p_3/q_3,p_4/q_4)$ by taking a copy
of $X=W_2(p_1/q_1, p_2/q_2)-int(N_2)$ and a copy of
$X'=W_2(p_3/q_3,p_4/q_4)-int(N_2)$ and glue the two boundaries by a
certain fixed point free, orientation reversing map
\[
\Psi : \partial X\longrightarrow \partial X'.
\]
To define $\Psi$ we use the identification of $\partial X$ and
$\partial X'$ with $G\times S^1$ described earlier, where the curves
$s_1,t_1,s_2,t_2$ lie on $G_2\times \{p\}$ as above. Furthermore we
fix the orientation reversing fixed point free involution $\psi:
G\longrightarrow G$ defined earlier. Then $\Psi$ is given by
$$\Psi(g,\theta) = (\psi(g), \theta)$$
for $g \in G$, $\theta \in S^1$.
Now let
\[
P=P(p_1/q_1, p_2/q_2,p_3/q_3,p_4/q_4) = X\cup _\Psi X'.
\]

It follows from the construction that if $p_1/q_1=p_3/q_3$ and
$p_2/q_2=p_4/q_4$ then $X$ and $X'$ are copies of the same 4-manifold,
and the map $(x,x')\rightarrow (x',x)$ extend to a fixed point free,
orientation preserving diffeomorphism $\tau(p_1/q_1, p_2/q_2)$ with
$\tau(p_1/q_1,p_2/q_2)\circ\tau(p_1/q_1,p_2/q_2) ={\rm Id}_P$.

\begin{defn}\label{def:DefOfXn}
  Given a positive integer $n\in \N$, let $X_n$ denote the
  four-manifold $P(-1,-n,-1,-n)$ constructed above.
\end{defn}
\begin{thm}\label{thm:involution1}
  The four-manifold $X_n$ admits a smooth, fixed point free,
  orientation preserving involution $\iota _n$.
\end{thm}
\begin{proof}
  The involution $\tau(p_1/q_1, p_2/q_2)$ with the given parameters $(-1,-n)$
  has the required properties.
  \end{proof}

\subsection{Fundamental group calculations}
\label{ssec:topological}
In this subsection we determine the topology of the manifolds
$X_n$. The computation of the fundamental group turns out to
be the most challenging:
\begin{thm}\label{thm:FundGrofXn}
  The four-manifold $X_n$ is simply connected.
\end{thm}

It is proved in \cite{BK} that $\pi_1(H\times K-(T_1\cup T_2), x_0)$
is generated by $x,y,a,b$ and these generators satisfy the relations
$[x,a]=1$, $[y,a]=1$. The following result on $\pi_1(V_0(-1,-n), x_0)$
follows immediately from the Seifert-Van Kampen theorem and the
description on pushoffs in Subsection~\ref{subsec:cons}.

\begin{lem}\label{path_V}
  Let $n$ be a positive integer.
  The fundamental group  $\pi_1(V_0(-1,-n), x_0)$ is generated $x,y,a,b$, and
  these generators satisfy the relations:
  \begin{itemize}
  \item $[x,a]=1$
  \item $[y,a]=1$
  \item $[b^{-1},y^{-1}]=x$
  \item $[x^{-1},b]^n = bab^{-1}$ (or equivalently $[b^{-1},x^{-1}]^n= a$).
  \end{itemize}
  Furthermore the inclusion map $\pi_1(V_0(-1,-n),x_0) \rightarrow
  \pi_1(V(-1,-n),x_0)$
  is surjective and the images of $[x,y]$ and $[a,b]$ are trivial. \qed  
\end{lem}

Recall the definition of $M_0$ from Equation~\eqref{eq:m0} and
consider $M_2=M_0\#2\cpkk$.
Fix an arc $\gamma$ in $M_2-int(N_2)$ from $x_1$ to $x_0$ with the
property that $\gamma$ lies inside ${\bf D}\times {\bf D}$.

\begin{lem}\label{path_M}
  We have  following path homotopies in $M_2- int(N_2)$:
\begin{itemize}
  \item $s_1$ is path homotopic to $\gamma\ast x\ast {\overline \gamma}$,
  \item $t_1$ is path homotopic to $\gamma\ast y\ast {\overline \gamma}$,
  \item $s_2$ is path homotopic to $\gamma\ast a\ast {\overline \gamma}$, and
  \item $t_2$ is path homotopic to $\gamma\ast b\ast {\overline \gamma}$.
\end{itemize}
\end{lem}

\begin{proof}
First note that the exceptional spheres in $M_2$ intersect $G_2$
transversally once.  It follows that there is a disk in $M_2-int(N_2)$
whose boundary is a meridional circle to $G_2$. It then follows from
the Seifert-Van Kampen theorem that
  $$i_\ast :\pi_1(M_2-int(N_2), x_1)\longrightarrow \pi_1(M_2,x_1)$$
  is an isomorphism. It follows that it is sufficient to prove the existence of path homotopies in $M_2$.

  For the first two bullet-points we can work in $M_{x,y}\subset M_2$,
  that is, $T^2\times {\bf D}$ blown up at $k_1$.  Both loops lie in
  $M_{x,y}$ and are based at $x_1$. Since $M_{x,y}$ has
  Abelian fundamental group, and the loops are homologous,
  it follows that they are also path homotopic.
  A similar argument inside $M_{a,b}\subset M_2$ (given as ${\bf
    D}\times T^2$ blown up at $k_2$) finishes the argument for the
  last two bullet-points.
  \end{proof}

Combining Lemma~\ref{path_V} and Lemma~\ref{path_M} we have the following:

\begin{lem}\label{path_W}
The fundamental group $\pi_1(W_2(-1,-n)-int(N_2), x_1)$ is generated by
$s_1,t_1,s_2,t_2$, and these generators satisfy the relations:
  \begin{itemize}
  \item $[s_1,s_2]=1$
  \item $[t_1,s_2]=1$
  \item $[t_2^{-1},t_1^{-1}]=s_1$
  \item $[t_2^{-1},s_1^{-1}]^n= s_2$
  \item $[s_1,t_1]=1$
  \item $[s_2,t_2]=1$.
  \end{itemize}
\end{lem}

\begin{proof}
  According to Lemma~\ref{path_M} it suffices to prove the
  corresponding results for $x,y,a,b$ in
  $\pi_1(W_2(-1,-n)-int(N_2),x_0)$.  Using the exceptional disks as
  before, we see that the map induced by the inclusion $i: W_2(-1,-n)-
  int(N_2)\rightarrow W_2(-1,n)$ is an isomorphism on $\pi_1$.  Since
  the blow-up operation does not change the fundamental group, the
  application of Lemma~\ref{path_V} on $\pi_1(V(-1,-n), x_0)$
  concludes the proof.
  \end{proof}

The map $\psi$ induces a map $\psi _*\colon \pi _1(G, p)\to \pi _1(G,p')$,
where the two base points $p$ and $p'$ are opposite points of
the circle $\mu$ we get by intersecting $G$ with the plane to which $r$
reflects, cf. Figure~\ref{fig:maps}. (Note that $\mu$ can be also viewed
as a loop based at $p$.)

Connecting $p'$ to $p$ by one of the semi-circles of $\mu$, we can identify
$\pi _1(G, p')$ with $\pi _1(G, p)$ in such a way that when composing
$\psi _*$ above with this identification, we get that
\begin{equation}\label{eq:action}
  s_1\mapsto t_2^{-1}\mu , \qquad t_1 \mapsto \mu ^{-1} s_2^{-1},
  \qquad s_2\mapsto \mu ^{-1}t_1^{-1},
  \qquad t_2\mapsto s_1^{-1}\mu.
  \end{equation}

With this preparation in place, we are ready to return to the
proof of Theorem~\ref{thm:FundGrofXn}.
    
%\begin{prop}\label{prop:fundamentalgp}
%  The four-manifolds $P(-1,-n,-1,-n)$ are simply connected.
%\end{prop}

\begin{proof}[Proof of Theorem~\ref{thm:FundGrofXn}]
  Fix a positive integer $n$, and let $X_n=P(-1,-n,-1,-n)$ (as in
  Definition~\ref{def:DefOfXn}).  In the copy $X= W_2(-1,-n)-int(N_2)$
  the generators of the fundamental group will be denoted by $s_1,
  t_1, s_2, t_2$ while in the copy $X'$ the same generators will be
  denoted by $S_1, T_1, S_2,T_2$.

  Relations among these generators are given above in
  Lemma~\ref{path_W} (the same relations should be repeated
  for the generators $S_1, T_1, S_2, T_2$). By the Seifert-Van Kampen
  theorem the above eight elements generate $\pi _1(X_n, p)$, which
  are subject to the above relations. In addition, the gluing map
  $\psi$
  provides further relations among these generators. Indeed, the
  action derived from the map $\psi$ on $\pi _1(G, p)$ described above
  implies that we need to identify

\begin{itemize}
  \item $s_1$ with $T_2^{-1}\mu$,
  \item $t_1$ with $\mu^{-1} S_2^{-1}$,
  \item $s_2$ with $\mu^{-1} T_1^{-1}$, and
  \item $t_2$ with $S_1^{-1}\mu$.
\end{itemize}
    Observe that $\mu $ is equal to the commutator $[s_1,t_1]$, which
    is trivial in the fundamental group $\pi _1(X_n,p)$, hence we can
    ignore it in the above list of identifications. The relation
    $[s_1,s_2]=1$ therefore implies that $[T_2,T_1]$ is trivial,
    implying that $S_1$ is also trivial, which immediately implies
    that $t_2$ is trivial. In addition, the relation $[T_1^{-1},
      S_1^{-1}]^n=S_2$, together with $S_1=1$ implies $S_2=1$ (and so
    $t_1=1$). But since $t_1=t_2=1$, the relation $[t_2^{-1},
      t_1^{-1}]s_1^{-1}=1$ implies $s_1=1$ (which also shows
    $T_2=1$). Finally, the triviality of $s_1$ and
    $[t_2^{-1},s_1^{-1}]^n=s_2$ implies the triviality of $s_2$, which also
    means that $T_1=1$, ultimately verifying the theorem.
\end{proof}

In order to complete the topological characterisation of $X_n$, we
need to determine its Euler characteristic $\chi (X_n)$ and signature
$\sigma (X_n)$.
\begin{prop}
  We have $\chi (X_n)=8$ and $\sigma (X_n)=-4$, consequently
  $b_2^+(X_n)=1$, $b_2^-(X_n)=5$.
  \end{prop}
\begin{proof}
  We can construct $X_n$ by first making the fiber sum operation between
  two copies of  $T^4\#2\cpkk$ to construct a four-manifold $U$ and then doing the four torus surgeries.
  Since the fiber sum is along a genus-2 surface we have
  $\chi (U)=2\chi(T^4\#2\cpkk)+4=8$.
  Since the signature is additive under the fiber sum operation
  we have $\sigma (U)=-4$. As torus surgery does not
  change $\chi$ and $\sigma$ we have
  $\chi(X_n)=8$ and $\sigma(X_n)=-4$. The values of
  $b_2^{\pm}(X_n)$ follow at once.
\end{proof}
The application of Freedman's theorem~\cite{Fr} then implies
\begin{cor}\label{cor:homeob2=2}
  For $n \in \N$ the four-manifold $X_n$ is homeomorphic to
  $\cpk\# 5\cpkk$. \qed
  \end{cor}

\subsection{Computation of Seiberg-Witten invariants for $X_n$.}
\label{ssec:computeSW}
We devote this subsection to the computation of $SW_{X_n}$,
  providing the proof of the following

\begin{thm}\label{thm:X-SW}  
  The four-manifold $X_n$ admits two Seiberg-Witten
  basic classes $\pm L$, and $SW_{X_n}(\pm L) = \pm n^2$.
\end{thm}

We work with the more general family
$P(p_1/q_1,p_2/q_2,p_3/q_3,p_4/q_4)$.
In fact $q_i=0$ is a special case, and it is equivalent to not doing
surgery along the corresponding torus, as
in this case the boundary of the gluing disk in $D^2\times T^2$ is the
same as $\pm 1$ times the meridian.

Using this observation we will start with $U= P(1/0,1/0,1/0,1/0)$
which can be viewed as by doing a fiber sum of two copies of $T^4\# 2\cpkk$ 
along $G_2$.  We
compute the Seiberg-Witten invariants of $U$ first and then use
surgery formulas along $T^3$ to compute the invariants for
$P(-1,-n,-1,-n)$.

\begin{lem}\label{lem:Usymplectic}
  The four-manifold $U$ admits a symplectic structure $\omega _U$. 
\end{lem}

\begin{proof}
  On the first copy of $T^4\# 2\cpkk$ use a symplectic form $\omega$
  for which $G_2$ is symplectic.  On the second copy of $T^4 \#
  2\cpkk$ use $-\omega$ so that $G_2$ is symplectic with the opposite
  orientation.  Applying the symplectic normal connected sum operation
  of Gompf~\cite{Gompf}, it follows that $U$ admits a symplectic
  structure.
\end{proof}

As $b_2^+(U)>1$, Equation~\eqref{eq:symp} implies that the Seiberg-Witten
function $SW_U$ is non-trivial, in
particular $SW_U (\pm c_1(U, \omega _U))=\pm 1$. Indeed,

\begin{lem}\label{lem:U}
  The symplectic four-manifold $U$ has two basic classes, which are
  $\pm c_1 (U, \omega _U)$.
\end{lem}
\begin{proof}
A Mayer-Vietoris calculation shows that $H_2(U; \Z)=\Z ^{14}$; a
basis of this homology group can be identified as follows. We have
four hyperbolic pairs of tori given by the four Lagrangian tori and
their dual tori in $T^4$; denote these by $(d_i, D_i)$ for $1\leq i \leq 4$
(see Figure~\ref{fig:surfaces}).

\begin{figure}[htb]
\begin{center}
\setlength{\unitlength}{1mm}
 \includegraphics[height=10cm]{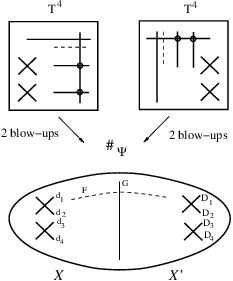}
\end{center}
\caption{\quad In the top row the two copies of $T^4$ with the
  intersecting section and fiber (and further tori) are shown.  In the
  bottom row the fiber connected sum of the two copies of
  $T^4\#2\cpkk$, together with the 0-tori $d_i, D_i$ ($i=1,\ldots ,
  4$), and the genus-2 surfaces $F$ and $G$ are depicted.}
  \label{fig:surfaces}
\end{figure}

Let $x$ denote the homology class of $G_2$ and choose a torus $F_1$ in
the first copy of $T^4\# 2\cpkk$ that represents the homology class
$[T^2\times pt]$ and intersects $G_2$ transversally at one
point. Choose a similar torus $F_2$ from the second copy as
well. Puncturing these surfaces and connecting them in the gluing
region by an annulus we get an embedded genus-2 surface $F$
representing a homology class $y$ with self-intersection 0 and $x\cdot
y = 1$.  This gives an additional hyperbolic pair in $H_2(U,\Z)$.

By gluing the exceptional disks in the first copy of $T^4\# 2\cpkk$ to
$F_2$ we get embedded tori $Q_1$, $Q_2$ with self-intersection $-1$.
Similarly by gluing $F_1$ to the exceptional disks in the second copy
of $T^4\# 2\cpkk$ we get embedded tori $Q_3$, $Q_4$ with
self-intersection $-1$. Let $q_i$ represent the corresponding homology
classes, where we orient the tori so that $x\cdot q_i=1$ for
$1\leq i\leq 4$. Then the homology classes $y-q_i$ span a four-dimensional
negative
definite subspace of $H_2(U; \Z )$ while the perpendicular subspace in
$H_2(U,\Z)$ is spanned by the above 5 hyperbolic pairs.

Using the adjunction inequality we see that any Seiberg-Witten basic class $K$
of $U$ 
evaluates as zero on all the 0-tori $d_i, D_i$.  By using the
adjunction inequality again, $K$ evaluates $\pm 2$ or 0 on $x =[G]$
and $y = [F$]. In order for $K$ to give rise to a moduli space with
non-negative formal dimension $\frac{1}{4}(K^2-3\sigma (U)-2\chi
(U))=\frac{1}{4}(K^2-4)$, we need that $K^2\geq 4$, hence
$K(x)= K(y) = \pm 2$.

Suppose that $K(x)=2$.  Since $K$ is characteristic,
we have that $K(y-q_i)$ is odd,
and $K^2 \geq 4$ implies that $K(y-q_i)=\pm 1 $ for all $1\leq i\leq
4$.  We will show that $K(y-q_i)=1$ for all $i$. To that end note that
$q_i$ is represented by a torus with self intersection $-1$ that
intersects a genus-2 surfaces
representing $x$ at a single transverse
point. Now $x+q_i$ is represented by a genus-3 surface $V$. Since
$[V]^2=(x+q_i)^2 = 1$ the adjunction inequality $K(V)+[V]^2 \leq
2g(V)-2=4$ shows that $K(V)=K(x+q_i)\leq 3$. Now $K(x)=2$ implies
$K(q_i) \leq 1$. Since $K(y)=2$ we have $K(y-q_i) \geq 1$.  This
together with $K(y-q_i)= \pm 1$ shows that $K(y-q_i)=1$. It follows
that in the case $K(x)=2$ the value of $K$ is determined on all
generators. The same argument works for $K(x)= -2$.

It follows that there are at most two basic classes on $U$.  On the
other hand, by Taubes' theorem \cite{Taub1} we know that
$\pm c_1(U, \omega _U)$ are basic
classes of $U$, completing the argument.
\end{proof}

We will also need the following:

\begin{lem}\label{lem:vanishing}
  Let $P = P(p_1/q_1,p_2/q_2, p_3/q_3, p_4/q_4)$.  If any of the $p_i$
  is equal to $0$ then all Seiberg-Witten invariants of $P$ vanishes.
\end{lem}

Before the proof of this lemma, we make a topological observation
about the four-manifold $W_2(0, p_2/q_2)$.

\begin{lem}\label{lem:spheres}
  The four-manifold  $W_2(0,p_2/q_2)$
  of Equation~\eqref{eq:defofw2} contains a smoothly embedded sphere $S_1$ of
  self-intersection $0$ that represents the homology class of
  $[T^2\times pt]$ and $S_1$ intersects $G_2$ transversally once.
  Similarly $W_2(p_1/q_1,0)$ contains a smoothly embedded sphere $S_2$
  of self-intersection $0$ that represents the homology class of
  $[pt\times T^2]$ and $S_2$ intersects $G_2$ transversally once.
\end{lem}
  
\begin{proof}
  For the construction of $S_1$ recall that the torus $T_1$ is given
  by the equations $z_2=c_2$, $z_4 = c_4$. Choose a generic
  $w_3$. Then the torus $z_3=w_3$, $z_4=c_4$ intersects in $T^4$ the
  tubular neighborhood of $T_1$ in an annulus, and intersects the
  boundary torus $T^3_1$ in two circles: $z_2=c_2\pm \delta_1$,
  $z_3=w_3$, $z_4=c_4$. Both of these circles represent the framing
  $m_1$ so they bound disjoint embedded disks in the $D^2\times T^2$
  part that we glue to $T^3_1$ when doing surgery with slope
  $m_1$. Deleting the annulus and gluing in the two disjoint disks we
  get an embedded sphere $S_1$ in $V(0,p_2/q_2)$.  Choose $w_3$ so
  that $S_1$ is disjoint from the blow-up points and also the
  smoothing region. Then $S_1$ is also in $W_2(0,p_2/q_2)$ and
  intersects $G_2$ transversally in one point.  The sphere $S_2$ can
  be constructed by the same argument.
\end{proof}

\begin{proof}[Proof of Lemma~\ref{lem:vanishing}]
  Let $p_i=0$ for some $i$. Let $L$ be a characteristic class in
  $H_2(P; \Z)$ with $SW_P(L)\neq 0$.  In  case  $q_j=0$  for some $j \neq i$ we
  get a corresponding hyperbolic pair in
  $H_2(P; \Z)$ represented by smoothly embedded tori
  of square 0.  Using Lemma~\ref{lem:spheres} we also get a homology class
  $y \in H_2(P; \Z )$ with the property that $y$ is represented by a
  smoothly embedded surface of genus 1, $y^2=0$, $y\cdot [G_2]=1$. By
  the adjunction inequality $L$ evaluates on the hyperbolic pairs and
  also on $y$ trivially. Since the orthogonal complement of the
  subspace represented by these classes is negative semi-definite in
  $H_2(P,\Q)$, it follows that $L^2\leq 0$.  However for a
  Seiberg-Witten basic class $L$ (as the formal dimension of the
  corresponding moduli space is non-negative) we have the inequality
  $L^2\geq 2\chi(P)+3\sigma (P)$, and $2\chi(P)+3\sigma (P)=4$.
  \end{proof}

\begin{proof}[Proof of Theorem~\ref{thm:X-SW}]
  In order to compute of the Seiberg-Witten invariants of
  $X_n=P(-1,-n,-1,-n)$, we start from $U= P(1/0, 1/0, 1/0, 1/0)$ and
  the basic classes $\pm L_0=\pm c_1(U, \omega _U)$ of $U$. Then we do
  the four surgeries one after the other to go from $U$ to
  $P(-1,1/0,1/0,1/0)$, to $P(-1,-n,1/0,1/0)$, to $P(-1,-n,-1,1/0)$, and
  finally to $P(-1,-n,-1,-n)$.  During this process each surgery kills
  one hyperbolic pair in $H_2$ given by the surgery torus and its dual
  torus.  Observe that $L_0$ evaluates trivially on the surgery torus,
  and so it gives a unique characteristic class after the surgery.
  Let $L_1$, $L_2$, $L_3$ and $L_4$ denote the images of $L_0$. The
  adjunction inequality argument in Lemma~\ref{lem:U} works for these
  manifolds as well (the only difference is the number of hyperbolic
  pairs spanned by smoothly embedded tori). It follows that for
  characteristic classes different from $\pm L_i$ the Seiberg-Witten
  invariant vanishes for the four-manifold at hand. To compute the
  value of the Seiberg-Witten function on $\pm L_i$, we use the
  formula of Equation~\eqref{eq:formula} inductively, together with
  Lemma~\ref{lem:vanishing}. It follows that $SW(L_1)=\pm 1$,
  $SW(L_2)=\pm n$, $SW(L_3)=\pm n$ and finally $SW(L_4)=\pm n^2$
  finishing the calculation.
\end{proof}

%Having determined the Seiberg-Witten invariants of $X_n$, next
%we concentrate on its topology

\begin{cor}\label{cor:X}
  For $n \in \N$ the four-manifold $X_n = P(-1,-n,-1,-n)$ are
  homoeomorphic but not diffeomorphic to $\cpk \# 5\cpkk$.  The
  manifolds $X_n$ and $X_m$ are diffeomorphic if and only if $n=m$.
\end{cor}

\begin{proof}
  The homeomorphism follows from Corollary~\ref{cor:homeob2=2}.
  Recall that $\cpk \# 5\cpkk$ has vanishing (small perturbation)
  Seiberg-Witten invariants, while by Theorem~\ref{thm:X-SW} the
  Seiberg-Witten invariants of $X_n$ are non-trival.

  If $n\neq m$ then $SW_{X_n}$ and $SW_{X_m}$ are different by
  Theorem~\ref{thm:X-SW} (distinguished by their Seiberg-Witten values
  on the unique pair of basic classes), so $X_n$ and $X_m$ are not
  diffeomorphic.
  \end{proof}

\subsection{The proof of Theorem~\ref{thm:b2=2}}
\begin{proof}[Proof of Theorem~\ref{thm:b2=2}]
  Consider the four-manifold $X_n = P(-1,-n,-1,-n)$ constructed above,
  with the free involution $\iota _n= \tau(-1,-n)$. Define
  $X'_n=X_n/\iota_n$.  As $\iota _n$ acts free, it follows that $\pi
  _1(X'_n)=\Z /2\Z$ and by Theorem~\ref{thm:TopClassification}
  (resting on the homeomorphism classification of smooth
  four-manifolds with finite cyclic fundamental groups,
  \cite[Theorem~C]{HK}) $X'_n$ is homeomorphic to $Z_1\#2\cpkk$.
  Since $X_n$ is the universal cover of $X'_n$, by applying
  Corollary~\ref{cor:X} it follows that $X'_n$ is diffeomorphic to
  $X'_m$ if and only if $n=m$.  As $X_n$ admits a single pair of basic
  classes, and these classes have square equal to 4, the
  irreducibility of $X_n$, and therefore of $X_n'$, follows at
  once. This observation concludes the proof of the theorem.
  \end{proof}

\section{Exotic structures on $Z_1\# \cpkk$}
\label{sec:b2=1}
The construction of examples in this case will follow similar lines as
the construction of $X_n$ in Section~\ref{sec:b2=2} --- the
difference in the construction, however, will require an even more
involved argument in the computation of $\pi _1$.

We will again use the manifolds $V(p_1/q_1, p_2/q_2)$ described in the
previous section, and we will reuse the notations for the loops
$x,y,a,b$ in the codimension-$0$ submanifolds
$V_0(p_1/q_1,p_2/q_2)\subset V(p_1/q_1,p_2/q_2)$ and $M_0\subset
V(p_1/q_1,p_2,q_2)$. The difference is that now we will find a genus-2
submanifold in $M_0$ blown up only once. This will be done by finding
a braided torus representing $2[T^2\times pt]$ and a torus
representing $[pt\times T^2]$ that intersect transversally twice, as
it has been introduced in \cite{AP}.  Then we smooth one of the
intersection points and blow up the other one. Since we use one
blow-up we will denote this genus-2 submanifold by $G_1$, to
distinguish it from the construction in the previous section.  The
construction follows \cite{AP}, but we make additional choices when
parametrizing a closed tubular neighborhood of $G_1$ with $G_1\times
D^2$.

\subsection{The construction of $G_1$}
For the construction, let $0<\delta \ll \epsilon$, which we will
specify later.  The braiding of the torus will happen inside
$T^2\times D_{2\epsilon}(\sqrt{2}\epsilon+\delta)$.  In order to
specify the torus we need a monotone increasing smooth function
\[
g\colon [0,1]\longrightarrow [0,\pi]
\]
with $g([0,6\epsilon])=0$ and $g([1-2\epsilon,1])=\pi$.
The braided torus is given by the points $(z_1,z_2,z_3,z_4)$ and
$\theta\in \{ \pi/4, 5\pi/4 \}$ which satisfy
\[
(z_3,z_4) =(2\epsilon +\sqrt{2}\epsilon \cdot {\rm cos}(g(z_1)+\theta), 2\epsilon+
\sqrt{2}\epsilon\cdot 
     {\rm sin}(g(z_1)+\theta)).
     \]
It follows from the choice of $g$ that the intersection of the braided
torus with ${\bf D}\times {\bf D}$ is
\[
  ({\bf D}\times (\epsilon,\epsilon))\cup
  ({\bf D}\times (3\epsilon, 3\epsilon)).
  \]
We choose the other torus in the construction to be
$(\epsilon,\epsilon)\times T^2$.  The intersection points between the
two tori are $k_1=(\epsilon, \epsilon, \epsilon, \epsilon)$ and
$k_2=(\epsilon, \epsilon, 3\epsilon, 3\epsilon)$.

Let $W_1=W_1(p_1/q_1,p_2/q_2)$ denote the blow up of $V(p_1/q_1, p_2/q_2)$
at $k_2$, and let $M_1$ denote $M_0$
blown up at $k_2$; obviously $M_1 \subset W_1$.
(Once again, the index is supposed to indicate the number of blow-ups.)
We have two smoothly embedded
tori $F_1$ and $F_2$ in $M_1\subset W_1$ with self-intersections $-1$ that
intersect transversally in $k_1=(\epsilon,\epsilon, \epsilon, \epsilon)$.
Smooth their intersection point $k_1$ inside the region
$D_{\epsilon}(\epsilon/3) \times D_\epsilon(\epsilon/3)$
to get a smoothly embedded genus-2 surface $G_1$ with self-intersection 0.

Now we choose $\delta$ so that the $\delta$-neighborhood $N_1$ of
$G_1$ is diffeomorphic to $G_1 \times D^2$ and is disjoint
from $V_0$.  Let ${\bf D'}= D_\epsilon(2\epsilon)$; then the part of
$T_1\cup T_2$ that is in ${\bf D'}\times {\bf D'}$ corresponds to two
transversally intersecting closed disks.

Note that the fundamental group of
\[
({\bf D'} \times {\bf D'}-G_1)
\]
is isomorphic to $\Z$, and we
choose $\delta$ so small that there is a deformation retraction to
\[ 
({\bf D'}\times {\bf D'})- int(N_1);
\]
it follows that $\pi_1({\bf D'}\times {\bf D'}-int(N_1))\cong \Z$.  Note
that $N_1$ is inside the region $M_1$, and $N_1$ is disjoint from $V_0$.

%{\bf Loops on the $\Sigma$:}

\subsection{Constructing the four-manifolds $Y_n$}\label{subsec:Yconst}
The construction will be similar to the doubling operation described
in the previous section, where we constructed $X_n$.  In order to
construct the gluing, we will fix a particular diffeomorphism,
described in this subsection, between $N_1$ and $G_1\times D^2$. We
will also need this trivialization in order to compute the map
$i_\ast$ from $\pi_1(G_1\times pt)$ to
$\pi_1(W_1(p_1/q_1,p_2/q_2))$. This computation of $i_\ast$ will be
used in the next subsection to show that the manifolds $Y_n$ are
simply connected.

Let $\lambda_1, \lambda_2, \lambda_3, \lambda_4$ be circles in
$G_1$ given by intersecting $G_1$ with $z_1=0$, $z_2=0$, $z_3=0$
and $z_4=0$ respectively. Note that $\lambda_1$ and $\lambda_2$
intersect at $w_2=(0,0,\epsilon, \epsilon)$ while $\lambda_3$ and
$\lambda_4$ intersect at $w_1 = (\epsilon,\epsilon, 0,0)$.  Choose a
smoothly embedded arc $\gamma_0$ in $\Sigma\cap ({\bf D'} \times {\bf
  D'})$ that goes from $w_1$ to $w_2$ so that $\gamma_0$ intersects
each $\lambda_i$ only at $w_1$ or $w_2$.

\begin{figure}[htb]
\begin{center}
\includegraphics[height=6cm]{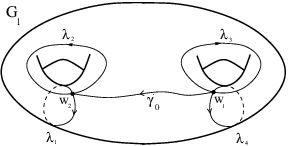}
\end{center}
\caption{\quad The loops $\lambda _i$ ($i=1, \ldots , 4$) and the path
$\gamma _0$ on $G_1$.}
\label{sigma}
\end{figure}

%{\bf Push-offs to $\partial N$:}
As a preparation to construct a trivialization of  $\partial N_1$,
we use  push-offs of $\lambda_i$ in $\partial N_1$, defined below.
Define $\nu_1$ as the set of points $(z_1,0,z_3,z_4)$ and
$\theta \in \{\pi/4, 5\pi/4\}$
where 
\[
(z_3,z_4) =(2\epsilon + (\sqrt{2}\epsilon+\delta) cos(g(z_1)+\theta), 2\epsilon+
(\sqrt{2}\epsilon+\delta) sin(g(z_1)+\theta)).
\]
Let $\nu_2$ be given by
\[
\{ z_1=0, z_3=z_4=\epsilon-(\delta/\sqrt{2})\},
\]
define $\nu_3$ as
\[
\{z_1=z_2=\epsilon-(\delta/\sqrt{2}), z_4=0\},
\]
and $\nu_4$ by the formula
\[
\{ z_1=z_2=\epsilon-(\delta/\sqrt{2}), z_3=0\}.
\]
Orient these circles so that in homology $[\nu_1]=2[x]$, $[\nu_2]=y$,
$[\nu_3]=a$ and $[\nu_4]=b$.

Note that $x_1=
(\epsilon-(\delta/\sqrt{2}),\epsilon-(\delta/\sqrt{2}),0,0)$, is the
intersection point between $\nu_3$ and $\nu_4$, while $x_2= (0,0,
\epsilon-(\delta/\sqrt{2}),\epsilon-(\delta/\sqrt{2}))$, is the
intersection point between $\nu_1$ and $\nu_2$.  From now on we will
treat $\nu_i$ as based loops, where $\nu_1$, $\nu_2$ are based at
$x_2$, and $\nu_3$, $\nu_4$ are based at $x_1$.  Let $\gamma_1$ be a
straight arc from $x_1$ to $x_0$ and $\gamma _2$ be a straight arc
from $x_2$ to $x_0$ inside ${\bf D'}\times {\bf D'}$.

Finally, choose an embedded arc $\gamma_3$ that is a push-off of 
$\gamma_0$ to $\partial N_1$ with the property that the homology
class of the loop
\[
\gamma_3\ast \gamma _2 \ast {\overline \gamma _1}
\]
based at $x_1$ is zero in $H_1(({\bf D'}\times {\bf D'})-int(N_1);
\Z)\cong \Z$. This can be achieved since the meridional circle around
$G_1$ that goes through $x_1$ generates $H_1(({\bf D'}\times {\bf
  D'})-int(N_1); \Z)\cong \Z$. We will also require that $\gamma_3$ is
smoothly embedded with the property that its tangent direction at
$x_1$ lies in the two-dimensional subspace spanned by the tangent
vectors of $\nu_3$ and $\nu_4$, and its tangent direction at $x_2$
lies in the two-dimensional subspace spanned by the tangent vectors of
$\nu_1$ and $\nu_2$.
Since $\pi_1(({\bf D'}\times {\bf D'})-int(N_1), x_1) \cong \Z$, it follows that
\[
\gamma_3\ast \gamma _2 \ast {\overline \gamma _1}
\]
is path homotopic to the constant loop $x_1$ in $({\bf D'}\times {\bf
  D'})-int(N_1)$.

Now we choose a parametrization of $N_1$ with $G_1\times D^2$ so
that all the loops $\nu_i$ and the arc $\gamma_3$ lie on some
$G=G_1\times \{ pt\}$, where $pt \in \partial D^2$. This can be achieved
by the following simple observation: Fix a parametrization
\[
f\colon N_1 \longrightarrow G_1 \times D^2
\]
and a $CW$
decomposition of $G_1$ with two 0-cells $w_1$ and $w_2$, five 1-cells
corresponding to $\lambda_i$ and $\gamma_0$, and a 2-cell, see
Figure~\ref{sigma}. Our
fixed push-off, composed with $f$ and projected to the $D^2$-factor
gives a map from the 1-skeleton to $\partial D^2 = S^1$. Since the
boundary of the $2$-cell goes through each 1-cells twice and with
opposite orientations, it follows that the map from the boundary of
the two-cell to $S^1$ is null-homologous, so it can be extended to the
two cell. This gives a section $G_1 \to G_1\times S^1$
that goes through all the push-offs as requested. Now we can use this section to give a new parametrization of $N_1$. Note that by the
earlier assumption on the choice of $\gamma _3$ we can also guarantee
that this parametrization is smooth.
In the normal direction to $G_1$ there is also a loop
$m$ based at $x_1$ given by
$\partial D_{\epsilon}(\delta)\times 0\times 0$, oriented counterclockwise.

We have now four elements in $\pi_1(G, x_1)$
represented by the loops $u_1 = \gamma_3\ast \nu_1 \ast {\overline
  \gamma_3}$, $v_1 = \gamma_3\ast \nu_2 \ast {\overline \gamma_3}$,
$u_2=\nu_3$ and $v_2=\nu_4$.  We will also use the separating circle
$\mu$ in $G$, where $[u_1,v_1]=\mu$ and
$[u_2,v_2]=\mu^{-1}$, and the normal loop $m$ that commutes in
$\pi_1(\partial N_1, x_1)$ with $u_1,v_1, u_2,v_2$. These elements
generate the group $\pi_1(\partial(N_1), x_1)$.

\begin{figure}[htb]
\begin{center}
\setlength{\unitlength}{1mm}
\includegraphics[height=6cm]{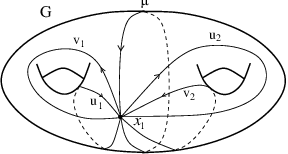}
\end{center}
\caption{\quad Embedded curves that represent the homotopy classes
  $u_1, v_1, u_2, v_2$ and $\mu$ in $\pi_1(G,x_1)$.}
\label{sigma2}
\end{figure}

Consider two such manifolds $X=W_1(p_1/q_1,p_2/q_2)-int(N_1)$ and
$X'=W_1(p_3/q_3,p_4/q_4)-int(N_1)$.  Using the trivializations
described above, the boundaries $\partial (X) $ and $\partial (X')$
are identified with $G_1\times S^1$ in such a way that the curves
$u_1,v_1,u_2,v_2$ lie on $G = G_1\times \{ pt\} $.  Then define
$$\Psi : G_1\times S^1\longrightarrow G_1\times S^1$$
as in the previous section, by $\Psi(g,\theta)=(\psi(g),\theta)$,
with the understanding that in the definition of $\psi$ we
use the curves representing $u_1$, $v_1$, $u_2$, $v_2$ instead of $s_1$, $t_1$, $s_2$ and $t_2$, see
Figure~\ref{sigma2}.

Now define the manifolds
$Q(p_1/q_1,p_2/q_2,p_3/q_3,p_4/q_4)$
by
\[
Q(p_1/q_1,p_2/q_2,p_3/q_3,p_4/q_4)=
\]
\[
=W_1(p_1/q_1,p_2/q_2)-int(N_1)\cup _\Psi
W_1(p_3/q_3,p_4/q_4)-int(N_1).
\]

\begin{defn}\label{def:Y}
  For a positive integer $n\in \N$, let $Y_n$ denote the smooth
  four-manifold $Q(-1,-n,-1,-n)$.
\end{defn}

%\begin{thm}
%  The four-manifold $Y_n$ admits a free, orientation preserving
%  involution $\iota _n$.
%\end{thm}
%\begin{proof}
%  As in the proof of Theorem~\ref{thm:involution1},
%  the involution $\tau(p_1/q_1, p_2/q_2)$ with the given parameters $(-1,-n)$
%  has the required properties.
%  \end{proof}

\subsection{Fundamental group calculations}
In this section we will show that $Y_n$ are simply connected.  A key
ingredient is the computation of the map $i_\ast$ from $\pi_1(G,x_1)$
to $\pi_1(M_1-int(N_1),x_1)$.  We start with some observations.

\begin{lem}\label{lem:spaces}
  The following spaces lie in $M_1-int(N_1)$:
\begin{itemize}
\item $({\bf D}- int(D_{\epsilon}(\delta)))\times S^1\times 0$,
\item $({\bf D} - int(D_{\epsilon}(\delta)))\times 0\times S^1$,
\item $0\times S^1\times ({\bf D}- int(D_{2\epsilon}(\sqrt{2}\epsilon+\delta))$,
\item $S^1\times 0\times ({\bf D}- int(D_{2\epsilon}(\sqrt{2}\epsilon+\delta))$, and
\item   $(T^2-int (D_{\epsilon}(\delta)))\times 0\times 0$. \qed
\end{itemize}
\end{lem}

\begin{lem}\label{path}
  We  have the following path homotopies in $M_1-int(N_1)$:
  \begin{itemize}
  \item ${\overline \gamma_1}\ast \nu_3 \ast \gamma_1$ is path homotopic to $a$,
  \item ${\overline \gamma_1}\ast \nu_4 \ast \gamma_1$ is path homotopic to $b$,
  \item ${\overline \gamma_2}\ast \nu_2 \ast \gamma_2$ is path homotopic to $x$,
  \item ${\overline \gamma_2}\ast \nu_1 \ast \gamma_2$ is path
    homotopic to $x^2[a,b]$, and
  \item ${\overline \gamma_1}\ast m \ast \gamma_1$ is path homotopic to  $[x,y]$.
  \end{itemize}
 \end{lem}

\begin{proof}
  All the identities above concern loops based at $x_0$. For the first
  identity note that both loops are in the subspace $({\bf D}-
  int(D_{\epsilon}(\delta)))\times S^1\times 0$ which is by 
  Lemma~\ref{lem:spaces} lies in $M_1-int(N_1)$.
  Since this space has Abelian fundamental
  group, and the loops are homologous, it follows that they are
  homotopic. So they are also homotopic in $M_1-int(N_1)$.

The same reasoning works for the second and the third identities.
For the fourth identity we can still use the fact that fundamental group of
$S^1\times 0\times ({\bf D}-
int(D_{2\epsilon}(\sqrt{2}\epsilon+\delta))$ is Abelian to show that
${\overline \gamma_2}\ast \nu_1 \ast \gamma_2$ is path homotopic in
$M_1-int(N_1)$ to $x\ast x \ast \lambda$, where $\lambda$ is a loop in
$$0\times 0\times {\bf
  D}-int(D_{2\epsilon}(\sqrt{2}\epsilon+\delta))$$ that is based at
$(0,0,0,0)$ and goes around $0\times 0\times 2\epsilon \times
2\epsilon$ counterclockwise once.  Since $$0\times 0\times (T^2-
int(D_{2\epsilon}(\sqrt{2}\epsilon+\delta)))$$ is in $M_1-int(N_1)$,
it follows that $\lambda$ is path homotopic in $M_1-int(N_1)$ to
$[a,b]$.  For the last statement we get the path homotopy inside
$(T^2-int (D_{\epsilon}(\delta)))\times 0\times 0$, which is in
$M_1-int(N_1)$.
\end{proof}

Moving the basepoint in computations involving fundamental groups is
always a delicate issue; we need to specify this move carefully. When
moving the basepoint from $x_0$ to $x_1$, in our subsequent arguments
we will use the arc $\gamma_1$ to identify $\pi_1(W_1(-1,-n)-int(N_1),
x_0)$ with $\pi_1(W_1(-1,-n)-int(N_1), x_1)$.

Let $s_1,t_1,s_2,t_2$ denote the images of $x,y,a,b$ under this
map. In particular $s_1$ is represented by $\gamma _1\ast x \ast
{\overline \gamma_1}$, $t_1= \gamma _1\ast y \ast {\overline
  \gamma_1}$, $s_2 = \gamma _1\ast a \ast {\overline \gamma_1}$, $t_2
= \gamma _1\ast b \ast {\overline \gamma_1}$.
As a corollary of Lemma~\ref{path_V}, we have the following relations in
$\pi_1(W_1(-1,-n)-int(N_1), x_1)$

\begin{itemize}
\item $[s_1,s_2]=1$
\item $[t_1,s_2]=1$
\item $[t_2^{-1}, t_1^{-1}]= s_1$
\item $[t_2^{-1},s_1^{-1}]^n= s_2$
\end{itemize}

%{\bf Computing the inclusion map:} 

\begin{lem}\label{lem:F}
Let $F$ denote the inclusion map
\[
F = i_\ast: \pi_1(\partial N_1, x_1) \longrightarrow \pi_1(W_1(-1,-n)-int(N_1), x_1).
\]
Then we have the following:
\begin{itemize}
\item The smallest normal subgroup of
  $\pi_1(W_1(-1,-n)-int(N_1),x_1)$ that contains  $s_1,t_1, s_2, t_2$ and $F(m)$
  is the entire group $\pi_1(W_1(-1,-n)-int(N_1),x_1)$.
\item $F(\mu)$ commutes with $F(u_1)$, $F(v_1)$, $F(u_2)$ and  $F(v_2)$.
\item $F(u_2)=s_2$, $F(v_2)=t_2$, $F(m)=[s_1,t_1]$.
\item $F(v_1) = t_1$.
\item $F(u_1) = s_1^2\cdot F(\mu)^{- 1}$.
\end{itemize}
\end{lem}

\begin{proof}
  The first statement follows from the Seifert-Van Kampen theorem,
  since $m$ is the meridional circle to $G_1$ and
  $\pi_1(W_1(-1,-n),x_1) $ is generated by $s_1,t_1, s_2, t_2$.

  For the second statement note that both  $\mu$ and $m$ are represented
  by  loops at $x_1$
  in ${\bf  D'}\times {\bf D'}-int(N_1)$, which has fundamental group $\Z$
  generated by $m$. It follows that $F(\mu)$ is a power of $F(m)$. 
As $m$ commutes with $u_1$, $v_1$, $u_2$ and $v_2$, the claim follows.
  
  The equations $F(u_2)=s_2$, $F(v_2)=t_2$, $F(m)=[s_1,t_1]$ follow immediately
  from Lemma~\ref{path}.
  To show $F(v_1) = t_1$, we use the third bullet-point of
  Lemma~\ref{path} to note that ${\overline \gamma_2}\ast \nu_2 \ast
  \gamma_2$ is path homotopic in $M_1-int(N_1)$ to $y$.
As by definition $\nu_2={\overline \gamma_3}\ast
v_1 \ast \gamma_3$, it follows that $t_1 = \gamma_1 \ast y \ast
   {\overline \gamma _1}$ is path homotopic to
   \[
   (\gamma_1\ast {\overline {\gamma _2}}\ast {\overline \gamma _3})
   \ast v_1  \ast (\gamma_3\ast \gamma _2 \ast {\overline \gamma _1}).
   \]
   To finish the argument note that the loop $\gamma_1\ast
   {\overline{\gamma _2}}\ast {\overline \gamma _3}$ is path homotopic
   to the constant loop $x_1$ (see the construction of $\gamma_3$ in
   Subsection~\ref{subsec:Yconst}).
   
  For the last equation we use from Lemma~\ref{path} that ${\overline
    \gamma_2}\ast \nu_1\ast \gamma_2$ is path homotopic in $M_1-int(N_1)$ to
  $x^2[a,b]$. Since $u_1 = {\gamma_3} \ast \nu_1 \ast {\overline \gamma
  _3}$, it follows that
  \[
  (\gamma_1\ast {\overline \gamma_2}\ast{\overline \gamma_3} ) \ast
  u_1 \ast (\gamma _3\ast \gamma _2\ast {\overline \gamma_1})
  \]
  is path homotopic to
  \[
  (\gamma_1\ast (x^2)\ast {\overline \gamma_1})\ast(\gamma_1\ast [a,b]\ast {\overline \gamma_1}).
  \]
The first expression is path homotopic to $u_1$, since $\gamma_1\ast
{\overline \gamma_2}\ast{\overline \gamma_3} $ is path homotopic to
the constant loop $x_1$.  The second expression is path homotopic to
$s_1^2[s_2,t_2]$. Since $[u_2,v_2]=\mu^{- 1}$ and $F(u_2)=s_2$,
$F(v_2)=t_2$ we get the desired equation.
\end{proof}

Now we have all the ingredients to determine the fundamental groups
of the manifolds $Y_n$ from Definition~\ref{def:Y}.
Recall that these manifolds
are constructed from
two copies $X$ and $X'$ of $W_1(-1,-n)-int(N_1)$, where we use the
orientation reversing fixed point free involution $\Psi$ to glue them together.
Note that in the fundamental group calculation here we use the curves representing
$u_1,v_1,u_2,v_2$ instead of $s_1,t_1,s_2,t_2$.

\begin{thm}\label{thm:fundQ}
  For a positive integer $n\in \N$, the oriented, smooth four-manifold
  $Y_n$ has trivial fundamental group.  Furthermore it
  admits an orientation preserving smooth, fixed point free involution, and
  has $b_2^+(Y_n)=1$ and $b_2^-(Y_n)=3$.
\end{thm}
\begin{proof}
  The determination of the Euler characteristic and signature are
  simple exercises from the same invariants of $W_1(-1,-n)$, and the
  free involution again comes from
  $\tau(-1,-n)$. The only
  part we need to check is the fundamental group computation. 

  Let $x_1$, $x_1'$ denote the basepoints in $X$ and $X'$.  In $X'$
  consider also $x_3$ that is antipodal to $x_1'$ on the separating
  circle $\mu$.  Connects $x_1'$ to $x_3$ on $\mu$ by an arc $\gamma$,
  and use this semi-circle $\gamma$   to identify the fundamental groups based at
  $x_1'$ and at $x_3$. This means that
  if
  \[
  \alpha \in \{u_1',v_1', u_2',v_2',\mu',m', s_1', t_1',s_2', t_2'\}
  \]
is one of the loops in $X'$ or $\partial (X')$ based at $x_1'$, the
corresponding element in $\pi_1(X',x_3)$ or $\pi_1(\partial (X'),
x_3)$ is given by ${\overline \gamma}\ast \alpha\ast \gamma$.  Let us
denote these elements in $\pi_1(X',x_3)$ and in $\pi_1(\partial (X'),
x_3)$ by capital letters -- except for $\mu'$.  For example,
${\overline \gamma}\ast u_1' \ast \gamma$ represents $U_1$; we denote
${\overline \gamma}\ast \mu'\ast \gamma$ by $\mu _1$.

Let  $F_1$ and $F_2$ denote the two inclusion maps
\[
F_1= i_\ast : \pi_1(\partial(X), x_1) \longrightarrow \pi_1(X, x_1)
\]
\[
F_2= i_\ast : \pi_1(\partial(X'), x_3) \longrightarrow \pi_1(X', x_3).
\]

These maps were computed earlier, see Lemma~\ref{lem:F}. In fact $F_1$ agrees with $F$, and in $F_2$
we use the generators with capital
letters and $\mu_1$ instead of $\mu$.

Note that $\psi$ maps $x_1$ to $x_3$.
In the Seifert-Van Kampen theorem we need $F_1$ and
\[F_3 = F_2\circ \psi_\ast.
\]
Note that $\psi_\ast(\mu)= \mu_1$.  Using our earlier computation on
the action of $\psi$ (compare also Figure~\ref{sigma2}), we have
\begin{itemize}
\item $\psi_\ast(u_1) = V_2^{-1}\mu_1$
\item $\psi_\ast(v_1) = \mu_1^{-1}U_2^{-1}$
\item $\psi_\ast(u_2) = \mu_1^{-1}V_1^{-1}$
\item $\psi_\ast(v_2) = U_1^{-1}\mu_1$.
\end{itemize}
%Let $Y_n$ denote the four-manifold $Q(-1,-n,-1,-n)$.
Let $F_4$ denote the inclusion map
\[
F_4: \pi_1(\partial(X), x_1)\longrightarrow \pi_1(Y_n, x_1).
\]
Our first goal is to show that
\[
F_3([v_1,v_2]) =1.
\]
  Since $F_2(\mu_1)$ commutes with $F_2(U_1)$ and with $F_2(U_2)$,
  it is enough to check that
  $F_2(U_1) = (S_1)^2\cdot F_2(\mu_1^{- 1})$ and
  $F_2(U_2) = S_2$ commute;
this follows since $S_1$ and $S_2$ commute.
It implies that  $F_4([v_1,v_2])=1$.

Using $[t_2^{-1},t_1^{-1}]=s_1$, $F_1(v_1)=t_1$, $F_1(v_2) = t_2$ and
$F_4([v_1,v_2])=1$, we get that the image of $s_1$ in $\pi_1(Y_n,
x_1)$ is trivial.  Now $[t_2^{-1},s_1^{-1}]^n=s_2$ kills the image
$s_2$ in $\pi_1(Y_n, x_1)$.  Since $F_1(u_2)=s_2$, the image of $u_2$
in $\pi_1(Y_n, x_1)$ is also trivial.  Since $\mu$ is the commutator
of $u_2$ and $v_2$ it follows that $F_4(\mu) =1$.  Since $F_1(u_1) =
s_1^2\cdot F_1(\mu^{- 1})$ we get $F_4(u_1)=1$.

We have shown that the loops representing $s_1$, $s_2$, $u_1$ and
$u_2$ are null-homotopic in $Y_n$. It follows by symmetry that the
loops $S_1$, $S_2$, $U_1$ and $U_2$ are also null-homotopic.  Since
$\mu_1$ is also null-homotopic, it follows that $F_3(v_1)$, $F_3(v_2)$
map to trivial elements in $\pi_1(Y _n, x_1)$.  Since $F_1(m) =
[s_1,t_1]$ it follows that the image of $m$ in $\pi_1(Y_n,x_1)$ is
also trivial.
   Since $u_1$, $v_1$, $u_2$, $v_2$, $s_1$, $S_1$ and  $m$
   normally generate $\pi_1(Y_n, x_1)$, the proof is complete.
   \end{proof}
  
\subsection{Computing the Seiberg-Witten invariants of $Y_n$}
For the Seiberg-Witten computations we will need the following
analogue of Lemma~\ref{lem:spheres}:
\begin{lem}\label{lem:spheres2}
  $W_1(0,p_2/q_2)$ contains a smoothly embedded sphere $S_1$ of
  self-intersection $0$ that represents the homology class of
  $[T^2\times pt]$ and $S_1$ intersects $G_1$ transversally once.
  $W_1(p_1/q_1,0)$ contains a smoothly embedded sphere $S_2$ of
  self-intersection $0$ that represents the homology class of
  $[pt\times T^2]$ and $S_2$ intersects $G_1$ transversally twice.
\end{lem}

\begin{proof}
  The construction of the spheres $S_1$ and $S_2$ in $V_0(0,p_2/q_2)$,
  $V_0(p_1/q_1,0)$ is the same as in Lemma~\ref{lem:spheres}. The only
  difference is that $S_2$ intersects $G_1$ twice.
 \end{proof}

\begin{thm}\label{thm:basicQ}
  $Y_n$ has two Seiberg-Witten basic class $\pm K$, and
  $SW_{Y_n}(\pm K) = \pm n^2$.
\end{thm}

\begin{proof}
The proof is similar to the one given in the previous section.  Let $R
= Q(1/0,1/0,1/0,1/0)$, which is given as a fiber sum of two copies of
$T^4\# \cpkk$.  Once again, $R$ is symplectic, so $SW_R(\pm
c_1(R,\omega _R))=\pm 1$.

Next we claim that $R$ has at most two Seiberg-Witten basic classes to
finish computing the Seiberg-Witten invariants of $R$. Now $H_2(R;
\Z)= \Z^{12}$ with 4 hyperbolic pairs spanned by tori with
self-intersection $0$, another hyperbolic pair $x,y$ represented by
genus-2 surfaces where $x$ is the homology class of $G_1$. Finally
there are two other classes $q_1$, $q_2$ in the orthogonal complement
of the hyperbolic pairs. These classes are represented by genus-2
surfaces with square $-1$, with $q_1\cdot q_2=0$, $q_i\cdot x=2$,
$q_i\cdot y=0$.  Furthermore these surfaces intersect $G_1$
transversally at two points.  Then $2y-q_1$, $2y-q_i$ span a
two-dimensional negative definite diagonal form, whose perpendicular
space is spanned by the 5 hyperbolic pairs.  Let $L$ be a basic class
of $R$. Then $L$ evaluates trivially on the 4 hyperbolic pairs spanned
by tori of self intersection 0.  Similarly $L(x)=L(y)=\pm 2$, and
$L(2y-q_i)= \pm 1$ follow from the adjunction inequalities and $L^2
\geq 2\chi(R)+3\sigma(R)=6$ from the non-negativity of the formal
dimension of the moduli space corresponding to $L$. Suppose now that
$L(x)=L(y)=2$; we show $L(2y-q_i)=1$ as follows:
Smoothing the two intersection points between $G_1$ and the embedded
genus-2 surface representing $q_i$, we get a genus-5 surface
$\Sigma_5$ that represents $x+q_i$. The adjunction inequality
  $$[\Sigma_5]^2+L(\Sigma_5) \leq 2g(\Sigma_5)-2$$
  shows that $(x+q_i)^2+L(x+q_i)\leq 8$. Since $(x+q_i)^2=3$
  we get $L(x+q_i) \leq 5 $ so $L(q_i)\leq 3$. Then $L(2y-q_i) \geq 1$ so
  $L(2y-q_i)=1$.

  For the next step we claim that if any of $p_i/q_i=0$, then all
  Seiberg-Witten invariants of $Q(p_1/q_1,p_2/q_2,p_3/q_3,p_4/q_4)$
  vanish. Let $x=[G_1]$.  In the case $i$ is odd, we can use the first
  part of Lemma~\ref{lem:spheres2} to get a torus of square 0 that
  intersects $x$ once. It follows that any basic class $L$ would be 0
  in a subspace of $H_2$ whose perpendicular space is negative
  semi-definite, contradicting $L^2 \geq 6$. In the case $i$ is even
  we can use the second part of Lemma~\ref{lem:spheres2} to get a
  surface of genus 2, representing a class $y_1$ with $y_1^2=0$ and
  $y_1\cdot x=2$. By the adjunction inequality we have $|L(x)|\leq 2$
  and $|L(y_1)| \leq 2$. Let $A$ denote the subspace of $H_2$ spanned
  by tori of self-intersection 0, and $B$ the subspace spanned by $x$
  and $y_1$. Then the perpendicular of $A\oplus B$ is negative
  definite.  Since $L$ evaluates trivially on $A$ it follows that $L^2
  \leq 4$, contradicting again $L^2 \geq 6$.

  The above vanishing result together with sequences of torus
  surgeries from $R$ to $Q(-1,-n,-1,-n)=Y_n$ finishes the calculation as
  discussed in the proof of Theorem~\ref{thm:X-SW}. 
\end{proof}

\begin{proof}[Proof of Theorem~\ref{thm:b2=1}]
  The manifolds $Y'_n$ are constructed by using the free involution
  $\iota _n=\tau (-1,-n)$ on $Y_n=Q(-1,-n,-1,-n)$.  Then $\pi_1(Y'_n)=
  \Z/2\Z$ follows immediately from Theorem~\ref{thm:fundQ}.  Using the
  same arguments as presented in the proof of Theorem~\ref{thm:b2=2},
  the irreducibility of $Y'_n$ and the exotic smooth structures on
  $Y'_n$ follow directly from the Seiberg-Witten invariants of $Y_n$
  in Theorem~\ref{thm:basicQ}.
\end{proof}

%\subsection{The proof of Theorem~\ref{thm:main}}

We conclude the paper with the proof of our main result:
\begin{proof}[Proof of Theorem~\ref{thm:main}]
  For a given $b_2>0$ and negative definite intersection form consider
  the $(b_2-1)$-fold blow-up of the four- manifolds $Y'_n$:
  \[
  A_n = Y_n'\# (b_2-1)\cpkk ;
  \]
  the universal cover of $A_n$ is
\[
Y_n\# (2b_2-2)\cpkk .
\]
Now the Seiberg-Witten invariants of this manifold in every chamber
take values in the subset
\[
\{0,\pm 1,\pm n^2, \pm n^2 \pm 1\},
\]
and there is a class and a chamber where the Seiberg-Witten invariant
is $\pm n^2$.  It follows that the double cover of $A_n$ and $A_m$ are
distinct for $n\neq m$, 
so $A_n$ and $A_m$ are not diffeomorphic either.  This proves the
existence of infinitely many smooth structures in the negative
definite case. In the positive definite case we use $-A_n$.
  \end{proof}

%\bibliography{biblio} \bibliographystyle{plain}

\end{document}